\documentclass{amsart}
\usepackage[latin2]{inputenc}
\usepackage{amsmath}
\usepackage{amsfonts}
\usepackage{amssymb}
\usepackage{graphicx}
\usepackage{amsthm}
\usepackage{marginnote}
\usepackage{mathtools}
\usepackage{mciteplus}
\theoremstyle{definition}
\newtheorem*{remark1}{Holomorphic examples}
\newtheorem*{remark2}{Quasiregular examples}
\newtheorem*{remark3}{Acknowledgements}
\newtheorem*{definition}{Definition}
\usepackage[colorlinks]{hyperref}
\hypersetup{colorlinks,linkcolor={blue},citecolor={blue},urlcolor={red}}\usepackage{amssymb}
\theoremstyle{plain}
\newtheorem{lemma}{Lemma}[section]
\theoremstyle{plain}
\newtheorem{theorem}{Theorem}[section]

\newcommand{\capac}{\operatorname{cap}}
\newcommand{\card}{\operatorname{card}}

\usepackage[left=3.90cm, right=3.90cm, top=3.00cm, bottom=3.00cm]{geometry}
\title{\textbf{Permutable Quasiregular Maps}}

\begin{document}
	\author{ATHANASIOS TSANTARIS}
	\address{School of Mathematical Sciences, University of Nottingham, Nottingham NG7 2RD.}
	\email{Athanasios.Tsantaris@Nottingham.ac.uk;}

	\begin{abstract}
		Let $f$ and $g$ be two quasiregular maps in $\mathbb{R}^d$ that are of transcendental type and also satisfy $f\circ g =g \circ f$. We show that if the fast escaping sets of those functions are contained in their respective Julia sets then those two functions must have the same Julia set. We also obtain the same conclusion about commuting quasimeromorphic functions with infinite backward orbit of infinity. Furthermore we show that permutable quasiregular functions of the form $f$ and $g=\phi\circ f$, where $\phi$ is a quasiconformal map, have the same Julia sets and that polynomial type quasiregular maps cannot commute with transcendental type ones unless their degree is less than or equal to their dilatation.
	\end{abstract}
	\maketitle
	\section{Introduction and Results}
	\let\thefootnote\relax\footnote{2020 \textit{Mathematics Subject Classification}. Primary 37F10; Secondary 30C65, 30D05}
	The general theory of iteration of holomorphic maps starts from the seminal work of Fatou \cite{fatou1919} and Julia \cite{Julia1918}. Both of them defined a partition of the complex plane in two sets. Those two sets today bear their names. They are the \textit{Fatou set}, $\mathcal{F}$, and the \textit{Julia set}, $\mathcal{J}$. In order to define them let us consider a holomorphic function $f:\mathbb{C}\to \mathbb{C}$ and denote by $ \left\lbrace  f^n\right\rbrace  $ the family of iterates of $f$, namely the family 
	$$\{ \underbrace{f\circ f \cdots \circ f}_\text{$n$ times}: n\in \mathbb{N}\}.$$
	Then the $\textit{Fatou set}$, $\mathcal{F}$, is defined as the set of points in a neighbourhood of which this family is normal and the \textit{Julia set}, $\mathcal{J}$, is defined as its complement. Fatou and Julia initially developed their theory for rational functions and later on Fatou \cite{fatou1926} also considered iteration of  transcendental entire functions. We refer to \cite{carleson-gamelin,milnor} for an introduction to rational iteration theory and to the survey \cite{survey} for the case of entire and meromorphic functions.\\\\
	Two holomorphic functions $f$ and $g$ are called \textit{permutable} or \textit{commuting} if they satisfy the equation $$f\circ g =g\circ f.$$
	A very old problem is to characterize all classes of functions that satisfy this equation. It turns out that commuting functions have a very similar dynamic behaviour. Indeed, one can prove that if $f$ and $g$ are permutable rational functions then they have the same Julia set. That was already shown by both Fatou \cite{fatou1923} and Julia \cite{julia1922} who used this fact to find all commuting rational functions that do not share an iterate (i.e. $f^m\not=g^n$ for all $n,m$) and do not have as their common Julia set the entire complex plane. Much later Eremenko in \cite{er2} developed their method further and managed to classify all commuting rational functions that do not share an iterate.  It is also worth mentioning here that Ritt in \cite{Ri1,Ri2} solved the same problem by using completely different methods.\\\\	
	For transcendental entire functions the problem is much harder and is still open. It is not even known if permutable  transcendental entire functions have the same Julia set or not. However Bergweiler and Hinkkanen \cite{B-H} in 1999, by introducing the so called \textit{fast escaping set} $\mathcal{A}$ $\left( f \right)$, managed to prove the following.
	\begin{theorem}\label{theorem berg}(\cite[Theorem 2]{B-H})
		Let $f$ and $g$ be permutable, transcendental entire functions such that $\mathcal{A}$$(f)\subset \mathcal{J}$$(f)$ and $\mathcal{A}$$(g)\subset \mathcal{J}$$(g)$ then $\mathcal{J}$$(f)$=$\mathcal{J}$$(g)$.
	\end{theorem}  
	Recently, Benini, Rippon and Stallard in \cite{BRS} managed to improve the above theorem and include some cases where $\mathcal{A}$$(f)\not\subset \mathcal{J}$$(f)$ and $\mathcal{A}$$(g)\not\subset \mathcal{J}$$(g)$. In particular they managed to show that two commuting functions will have the same Julia set if all their wandering Fatou components are multiply connected (such components are in the fast escaping set but not the Julia set). However the general case still remains open.\\\\
	In recent years there has been an increasing interest in developing an analogous theory to that of Fatou and Julia for quasiregular maps in $\mathbb{R}^d$. Quasiregular maps are a higher dimensional generalization of the analytic maps in the complex plane. Intuitively quasiregular maps have locally a bounded amount of distortion. This means that while analytic maps in the complex plane map infinitesimally small circles to circles, quasiregular maps send infinitesimally small spheres to ellipsoids of bounded eccentricity (see section 2 for a precise definition). In \cite{B-Nicks} Bergweiler and Nicks defined a Julia set for quasiregular maps of transcendental type (see section 2 for definition) and proved that it has many of the properties of the classical Julia set.\\\\
	There are examples of permutable functions in the quasiregular setting. So the natural thing to ask is: Do permutable quasiregular maps have a similar dynamic behavior? Can we generalize Theorem \ref{theorem berg} to quasiregular maps?\\\\
	In this paper we will adopt the definition of the Julia set from \cite{B-Nicks} and by using its properties we will prove the following generalization of Theorem \ref{theorem berg}.
	\begin{theorem}\label{my theorem}
		Let $f:\mathbb{R}^d\to \mathbb{R}^d$ and $g:\mathbb{R}^d\to \mathbb{R}^d$ be permutable, quasiregular maps of transcendental type such that $\mathcal{A}$$(f)\subset \mathcal{J}$$(f)$ and  $\mathcal{A}$$(g)\subset \mathcal{J}$$(g)$,  then ${\mathcal{J}}(f)={\mathcal{J}}(g)$.
	\end{theorem}
	The more general version of this theorem would be the analogous result to that of Benini-Rippon-Stallard. Unfortunately their proof relies heavily on the properties of the hyperbolic metric under holomorphic maps. Such an approach does not work in higher dimensions. 
	
	Another interesting question to ask is whether permutable quasiregular maps of polynomial type must have the same Julia set. However, this problem seems harder and is still open. On the other hand if $f,g$ are permutable, uniformly quasiregular maps of polynomial type (see section 2 for the definition) then ${\mathcal{J}}(f)={\mathcal{J}}(g)$ and the proof is almost the same as the one for rational functions in the complex plane which can be found in  \cite{Baker2}.
	Moreover, we can generalize a result of Baker \cite[Lemma 4.5]{Baker2} which deals with a special case and can be applied to quasiregular maps of polynomial or transcendental type.
	\begin{theorem}\label{baker generalization}
		Let $f:\mathbb{R}^d\to \mathbb{R}^d$ and $g:\mathbb{R}^d\to \mathbb{R}^d$ be permutable quasiregular maps. Assume that $\capac{\mathcal{J}}(f)>0, \hspace{1mm}\capac{\mathcal{J}}(g)>0 $ and $g=\phi\circ f$, where $\phi:\mathbb{R}^d\to\mathbb{R}^d$ is a quasiconformal map. Then ${\mathcal{J}}(f)={\mathcal{J}}(g)$.
	\end{theorem} 
	Note here that in the above theorem we assume that the capacity of the Julia sets of our functions is positive. It is conjectured that this always holds when the Julia set is infinite and thus we do not actually need this assumption. However, we can prove that this condition can be dropped if $g$ has a very specific form. Namely the following holds.
	\begin{theorem}\label{baker generalization 2}
		Let $f:\mathbb{R}^d\to \mathbb{R}^d$ and $g:\mathbb{R}^d\to \mathbb{R}^d$ be permutable quasiregular maps of transcendental type. Assume that $g=af+c$, where $a$ is a positive real number and $c$ is a constant in $\mathbb{R}^d$. Then ${\mathcal{J}}(f)={\mathcal{J}}(g)$.
	\end{theorem}
	We can also consider the case where $f,g$ are quasimeromorphic (see section 2 for the definition). Recently in \cite{WARREN2018} Warren defined the Julia set for quasimeromorphic maps of transcendental type. So it is interesting to ask whether something similar with Theorem~\ref{my theorem} holds in this case.  For quasimeromorphic maps we say that they are permutable if $f\circ g=g\circ f$ holds for points in $\mathbb{R}^d$ where both sides are defined.  In order to state our theorem in this setting let us introduce the concept of the \textit{backward orbit} of a point. Let $x\in\overline{\mathbb{R}^d}=\mathbb{R}^d\cup\{\infty\}$ then we define the backward orbit as \[O_f^{-}(x)=\bigcup_{n=0}^{\infty}f^{-n}(x).\]

	When studying the dynamics of meromorphic functions we usually divide them in two classes. The first one, and the most general one, is \[\mathcal{M}:=\{f: f \hspace{2mm}\text{is transcendental meromorphic and}\hspace{2mm}\card (O_f^{-}(\infty))=\infty\},\] while the other one is \[\mathcal{P}:=\{f: f \hspace{2mm}\text{is transcendental meromorphic and}\hspace{2mm}\card (O_f^{-}(\infty))<\infty\}.\] A typical example of a map in class $\mathcal{P}$ is $\frac{e^z}{z}$. The iteration theory and the methods of proof in those two classes are often quite different with class $\mathcal{P}$ being often closer to the class of transcendental entire functions instead. The situation is similar for quasimeromorphic maps. For functions in the analogous class $\mathcal{M}$ in higher dimensions we prove the following.

	\begin{theorem}\label{theorem2}
		Let $f:\mathbb{R}^d\to \overline{\mathbb{R}^d}$ and $g:\mathbb{R}^d\to \overline{\mathbb{R}^d}$ be permutable, quasimeromorphic maps of transcendental type with $\card (O_f^{-}(\infty))=\infty=\card(O_g^{-}(\infty))$, then  ${\mathcal{J}}(f)={\mathcal{J}}(g)$.
	\end{theorem}
	However, as is often the case, the method used to prove this theorem cannot be used in class $\mathcal{P}$. Let us also note here that it is highly non trivial to construct such maps in higher dimensions and until recently (see \cite{Warren2019}) this had not been done.

	It is also worth mentioning here that Baker in \cite[Theorem 1 p. 244]{baker1} proved that given an entire function $f$, which is either transcendental or polynomial of degree at least two, then there are only countably many entire functions $g$ that are permutable with $f$. We will give examples which show that this theorem cannot hold in the quasiregular case. To be more specific, by modifying an example given in \cite{B-H}, we are able to prove the following result.
	\begin{theorem}\label{examples}
		There exists an entire transcendental map $f$ that is permutable with uncountably many quasiregular maps $g:\mathbb{C}\to \mathbb{C}$. 
	\end{theorem}  
	
	A natural thing to ask here is the following question:   Is it possible for a polynomial to commute with an entire transcendental function? The answer to that was given independently by Baker in \cite{baker1958} and Iyer in \cite{iyer}.
	\begin{theorem}[Baker, \cite{baker1958} and Iyer, \cite{iyer}]
		Let $f,g :\mathbb{C}\to \mathbb{C}$ be permutable functions. Assume that $g$ is a polynomial of degree $n>1$. Then $f$ is also a polynomial.
	\end{theorem}
	This means that there is no polynomial commuting with a transcendental entire function unless the polynomial has degree one. This result has been generalized to meromorphic functions by Osborne and Sixsmith in \cite{Osborne2016a}.
	
	In the quasiregular setting we are able to prove the following theorem which can be seen as the analogy to that of Baker and Iyer in higher dimensions.
	
	\begin{theorem}\label{theorem1}
		Let $f:\mathbb{R}^d\to\mathbb{R}^d$ and $g:\mathbb{R}^d\to\mathbb{R}^d$ be permutable quasiregular maps. If $g$ is of polynomial type and $\deg g> K(g)$ then $f$ is also of polynomial type.
	\end{theorem}
	Thus a polynomial type quasiregular map can commute with a transcendental type map only if its degree is less than or equal to its dilatation.

	\textbf{Remark:} The condition $\deg g> K(g)$ for $g$ a polynomial type quasiregular map appears naturally in quasiregular dynamics, see for example \cite{Berg1,berg2013,FLETCHER2010}. It plays the same role as the condition $\deg g\geq2$ in holomorphic dynamics, when $g$ is a polynomial.\\

	The structure of the rest of this paper is as follows. Section 2 contains background material on quasiregular maps and capacity. In section 3 we prove Theorems \ref{my theorem} and \ref{baker generalization}. Section 4 contains the proof of Theorem \ref{baker generalization 2} while in section 5  we prove Theorem \ref{examples}. Lastly, in section 6 we prove Theorems \ref{theorem2}, \ref{theorem1} and in section 7 we provide some examples that help illustrate our theorems.  
	\begin{remark3}
		I would like to thank my supervisor Dr. Daniel Nicks for his patient guidance and for various valuable suggestions that greatly improved this paper. I am also grateful to Dr. Dave Sixsmith and the anonymous referee for their questions which led to Theorem \ref{theorem1}. Moreover, I would  like to thank the anonymous referee for his comments and suggestions which helped improve the exposition. 
	\end{remark3} 
	\section{Background on quasiregular maps and capacity}
	Here we will give a brief overview of the properties of quasiregular maps that we will need. For a more detailed treatment of quasiregular maps we refer to \cite{Rickman, vuorinen}. For a survey in the iteration of such maps we refer to \cite{Berg1}.\\\\
	If $d\geq 2$ and $G\subset \mathbb{R}^d$ is a domain, then for $1\leq p <\infty$ the  \textit{Sobolev space} $W^1_{p,loc}(G)$ consists of functions $f=(f_1,f_2,\cdots, f_d):G\to\mathbb{R}^d$ for which the first order weak partial derivatives $\partial_i f_j$ exist and are locally in $L^p$. A continuous map $f\in W^1_{d,loc}(G)$ is called \textit{quasiregular} if there exists a constant $K_O\geq 1$ such that \begin{equation}\label{quasi}
	\left|Df(x)\right|^d\leq K_O J_f(x) \hspace{2mm} a.e.,\end{equation} where $Df(x)$ denotes the total derivative, $$|Df(x)|=\sup_{|h|=1}|Df(x)(h)|$$ denotes the operator norm of the derivative, and $J_f(x)$ denotes the Jacobian determinant. Also let $$\ell(Df(x))=\inf_{|h|=1}|Df(x)(h)|.$$
	The condition that \eqref{quasi} is satisfied for some $K_O\geq1$ implies that $$K_I\ell(Df(x))^d\geq J_f(x),\hspace{2mm} a.e.,$$
	for some $K_I\geq 1$. The smallest constants $K_O$ and $K_I$ for which those two conditions hold are called \textit{outer dilatation} and \textit{inner dilatation} respectively. We call the maximum of those two numbers the dilatation of $f$ and we denote it by $K(f)$. We say that $f$ is $K$-quasiregular if $K(f)\leq K$, for some $K\geq 1$. We also say that $f$ is \textit{uniformly $K$-quasiregular} if all the iterates of $f$ are $K$-quasiregular. Quasiregular maps have many of the properties that holomorphic maps have. In particular, we will often use the fact that non-constant quasiregular maps are open and discrete.\\\\
	An important tool that we will need in order to define the Julia set of a quasiregular map is the capacity of a condenser. A condenser in $\mathbb{R}^d$ is a pair $E=(A,C)$, where $A$ is an open set in $\mathbb{R}^d$ and $C$ is a compact subset of $A$. The \textit{conformal capacity} or just \textit{capacity} of the condenser $E$ is defined as 
	$$\capac E=\inf_{u}\int_{A}|\nabla u|^ddm,$$
	where the infimum is taken over all non-negative functions $u\in C_0^{\infty}(A)$ which satisfy $u_{|C}\geq 1$ and $m$ is the $d$-dimensional Lebesgue measure.\\\\
	If $\capac (A,C)=0$ for some bounded open set $A$ containing $C$, then it is also true that $\capac (A',C)=0$ for every other bounded set $A'$ containing $C$;\cite[Lemma III.2.2]{Rickman}. In this case we say that $C$ has zero capacity and we write $\capac C=0$; otherwise we say that $C$ has positive capacity and we write $\capac C>0$. Also for an arbitrary set $C\subset \mathbb{R}^d$, we write $\capac C=0$ when $\capac F=0$ for every compact subset $F$ of $C$. If the capacity of a set is zero then this set has Hausdorff dimension zero \cite[Theorem~VII.1.15]{Rickman}. Thus a zero capacity set is small in this sense. It is also quite easy to see that for any two sets $S,B$ with $S\subset B$ if $\capac B=0$ then $\capac S=0$.\\\\
	A useful property of quasiregular maps is that they do not increase too much the capacity of condensers, namely the following theorem holds, which is known as the $K_I$ inequality, \cite[Theorem~II.10.10]{Rickman}.
	\begin{theorem}\label{KI ineq}
		Let $f:G\to \mathbb{R}^d$ be a nonconstant quasiregular map and $E=(A,C)$ a condenser in $G$, then $$\capac f(E)\leq K_I(f)\capac E.$$
	\end{theorem}
	\hspace{1mm}\\
	A quasiregular map $f:\mathbb{R}^d\to\mathbb{R}^d$ is said to be of \textit{transcendental type} if $\lim_{x\to\infty}f(x)$ does not exist and it is said to be of \textit{polynomial type} if this limit is $\infty$. Furthermore, if $G\subset \overline{\mathbb{R}^d}$, a non constant and continuous map $f:G \to \overline{\mathbb{R}^d}$ is called \textit{quasimeromorphic} if $f^{-1}(\infty)$ is discrete and $f$ is quasiregular in $G\setminus (f^{-1}(\infty)\cup\{\infty\})$. Rickman \cite{Rickman1980,rickman1985} has extended Picard's great theorem to quasiregular maps and shown that there exists a constant $q=q(d,K)$ such that if $f:\mathbb{R}^d\to \mathbb{R}^d$ is a $K$-quasiregular map of transcendental type then  there are at most $q(d,K)$ points that are taken finitely often by $f$. This means that if we define the exceptional set $E(f)$ for a $K$-quasiregular  map as the points with finite backward orbit, then $|E(f)|\leq q(d,K)$. \\\\
	In \cite{berg2013} Bergweiler developed a Fatou-Julia theory for  quasiregular self-maps of $\overline{\mathbb{R}^d}$, which include polynomial type quasiregular maps, and can be thought of as analogs of rational maps, while in \cite{B-Nicks} Bergweiler and Nicks did the same but for transcendental type quasiregular maps. Following those two papers we define the Julia set of $f:\mathbb{R}^d\to \mathbb{R}^d$, denoted ${\mathcal{J}}(f)$, to be the set of all those $x\in\mathbb{R}^d$ such that
	$$ \capac \left(\mathbb{R}^d\setminus \bigcup_{k=1}^\infty f^k(U)\right)=0
	$$	for every neighbourhood $U$ of $x$. Following \cite{NICKS2017} we call the complement of ${\mathcal{J}}(f)$ the \textit{quasi-Fatou set}, and we denote it by $QF(f)$ . We also want to define the Julia set for a quasimeromorphic map of transcendental type with at least one pole, $f:\mathbb{R}^d \to \overline{\mathbb{R}^d}$. This was done by Warren in \cite{WARREN2018} where he defined
	\[{\mathcal{J}}(f)=\Big\{ x\in\overline{\mathbb{R}^d}\setminus\overline{O^{-}_{f}(\infty)}:\card\left(\overline{\mathbb{R}^d}\setminus O^{+}_{f}(U_x)\right)<\infty \Big\}\cup\overline{O^{-}_{f}(\infty)},\] where $U_x$ is any neighbourhood of $x$ with $U_x\subset\overline{\mathbb{R}^d}\setminus\overline{O^{-}_{f}(\infty)}$ and $ O^{+}_{f}(U_x)=\bigcup_{n=0}^{\infty}f^n(U_x)$. 
	In particular if $f$ has an infinite backward orbit of infinity then ${\mathcal{J}}(f)=\overline{O^{-}_{f}(\infty)}$.\\\\
	Note that our definition of the Julia set evokes the blow-up property that the Julia set has  in complex dynamics. Also note that we do not assume anything about the normality of the family of iterates of $f$ in the quasi-Fatou set. For the motivation behind those definitions we refer to \cite{Berg1,berg2013}. Also let us mention that the definition of the Julia set using non-normality generalizes well for uniformly quasiregular maps and the two definitions are equivalent in this case. This is also true in the case of holomorphic maps in the complex plane.\\\\
	Finally, let us discuss the fast escaping set. The fast escaping set, as we have already mentioned, was first defined by Bergweiler and Hinkkanen in \cite{B-H}. Intuitively the  fast escaping set is the set of points that escape to infinity as fast as possible. In \cite{B-H} it is also proved that $${\mathcal{J}}(f)=\partial {\mathcal{A}}(f).$$
	
	Rippon and Stallard in their papers \cite{R-S2,R-S} gave two other equivalent definitions for the fast escaping set which are useful. They showed that
	\begin{equation}\label{third}
	{\mathcal{A}}(f)=\{z\in\mathbb{C}: \text{ there exists}\hspace{2mm}L\in\mathbb{N} \hspace{2mm}\text{such that} \hspace{2mm}f^{n+L}(z)\not \in T(f^n(D)), \hspace{2mm}\text{for all}\hspace{2mm} n\in \mathbb{N}\},\end{equation}
	where $D$ is any open disc meeting ${\mathcal{J}}(f)$ and $T(X)$ is the \textit{topological hull} of the set $X\subset \mathbb{C}$, in other words the union of $X$ with its bounded complementary components.\\\\
	The  fast escaping set of a quasiregular map, which was first described by Bergweiler, Drasin and Fletcher in \cite{B-D-F}, is defined in a very similar way to the complex case, namely 
	\begin{equation}\label{qfirst}
	{\mathcal{A}}(f)\vcentcolon=\{x\in\mathbb{R}^d: \text{ there exists}\hspace{2mm}L\in\mathbb{N} \hspace{2mm}\text{such that} \hspace{2mm}f^{n+L}(x)\not \in T(f^n(B(0,R))), \hspace{2mm}\text{for all}\hspace{2mm} n\in \mathbb{N}\},
	\end{equation}
	where $R>0$ is chosen so large that $$T(f^n(B(0,R)))\supset B(0,r_n)$$ and $r_n>0$ is a sequence that tends to $\infty$. Such an $R$ is guaranteed to exist by \cite[Lemma 5.1]{BFLM}.\\\\
	Also Bergweiler, Drasin and Fletcher gave two other equivalent definitions, in the same spirit as those for the complex case, which we omit here. Furthermore they proved that 
	\begin{theorem}\label{berg fast}
		Let $f:\mathbb{R}^d\to\mathbb{R}^d$ be a quasiregular map of transcendental type. Then ${\mathcal{A}}(f)$ is non-empty and every connected component of ${\mathcal{A}}(f)$ is unbounded.
	\end{theorem}
	For more details and the proof of this theorem we refer to \cite{B-D-F}.
	Unfortunately, in the quasiregular case it is still not known if ${\mathcal{J}}(f)=\partial {\mathcal{A}}(f)$. 
	But let us mention that the above equality is known to be true if $f$ does not grow too slowly and we always know that ${\mathcal{J}}(f)\subset\partial {\mathcal{A}}(f)$. We refer to \cite{B-F-N} for more details on this.
	
	\section{Proof of Theorems \ref{my theorem} and \ref{baker generalization}}
	In order to prove Theorem \ref{my theorem} we will need several lemmas. 
	
	\begin{lemma}\label{lemma 1}
		Let $f:\mathbb{R}^d\to \mathbb{R}^d$ and $g:\mathbb{R}^d\to \mathbb{R}^d$ be permutable quasiregular maps. Then $$g({\mathcal{J}}(f))\subset {\mathcal{J}}(f)\hspace{2mm}\text{and}\hspace{2mm}f({\mathcal{J}}(g))\subset {\mathcal{J}}(g).$$
	\end{lemma}
	\begin{proof}
		Take a $x_0\in{\mathcal{J}}(f)$ and take $U$ be a neighbourhood of $g(x_0)$. Name $V$ the component of $g^{-1}(U)$ which contains $x_0$. We know, by the definition of the Julia set, that 
		\begin{equation}\label{eq}
		\capac \left(\mathbb{R}^d\setminus\bigcup_{n=1}^{\infty}f^n(V)\right)=0.
		\end{equation}
		But since $f,g$ are permutable we have that $f^n(g(x))=g(f^n(x)), $ for all $ x\in V,$
		which implies that $f^n(x)\in g^{-1}(f^n(U)),\hspace{1mm}\text{for all}\hspace{1mm} x\in V.$
		Hence, 
		\begin{align*}
		\mathbb{R}^d\setminus\bigcup_{n=1}^{\infty}f^n(V)&\supset \mathbb{R}^d\setminus\bigcup_{n=1}^{\infty}g^{-1}(f^n(U)).
		\end{align*}
		Thus, by (\ref{eq}) and the fact that subsets of zero capacity sets have zero capacity we have that $$\capac \left(\mathbb{R}^d\setminus\bigcup_{n=1}^{\infty}g^{-1}(f^n(U))\right)=0.$$
		But since $\bigcup_{n=1}^{\infty}g^{-1}(f^n(U))=g^{-1}\left(\bigcup_{n=1}^{\infty}f^n(U)\right)$ and $g^{-1}(\mathbb{R}^d)=\mathbb{R}^d$ this implies that $\capac \left(g^{-1}\left(\mathbb{R}^d\setminus\bigcup_{n=1}^{\infty}f^n(U)\right)\right)=0.$
		Hence, by the $K_I$-inequality (Theorem \ref{KI ineq}) we will have that $$\capac \left(g\left(g^{-1}\left(\mathbb{R}^d\setminus\bigcup_{n=1}^{\infty}f^n(U)\right)\right)\right)=0.$$
		Since $g$ is a quasiregular self-map of $\mathbb{R}^d$ we know by Rickman's generalization of Picard's theorem that it omits at most a finite number of points. Thus $$g\left(g^{-1}\left(\mathbb{R}^d\setminus\bigcup_{n=1}^{\infty}f^n(U)\right)\right)=\mathbb{R}^d\setminus\left(\bigcup_{n=1}^{\infty}f^n(U)\cup\{a_1,a_2,\cdots, a_m\}\right),$$
		where $a_1,a_2,\cdots, a_m$ are the omitted values of $g$. 
		Hence, by using the well known fact that the union of capacity zero sets is also of zero capacity we will have that $\capac \left(\mathbb{R}^d\setminus\bigcup_{n=1}^{\infty}f^n(U)\right)=0.$
		Since $U$ was an arbitrary neighbourhood of $g(x)$, this implies that $g(x)\in {\mathcal{J}}(f).$\\\\
		For the other half of the theorem, the proof is completely analogous to this one with $f$ and $g$ changing roles.
	\end{proof}
	\begin{lemma}\label{lemma2}
		Let $f:\mathbb{R}^d\to \mathbb{R}^d$ and $g:\mathbb{R}^d\to \mathbb{R}^d$ be permutable quasiregular maps of transcendental type. Then $$g^{-1}\left({\mathcal{A}}(f)\right)\subset {\mathcal{A}}(f)\hspace{2mm}\text{and}\hspace{2mm}g^{-1}\left( \overline{{\mathcal{A}}(f)}\right)\subset \overline{{\mathcal{A}}(f)}. $$
		Also $$f^{-1}\left({\mathcal{A}}(g)\right)\subset {\mathcal{A}}(g)\hspace{2mm}\text{and}\hspace{2mm}f^{-1}\left( \overline{{\mathcal{A}}(g)}\right)\subset \overline{{\mathcal{A}}(g)}. $$
	\end{lemma}
	\begin{proof}
		Take $R_1>0$ so big that  $$T(f^n(B(0,R_1)))\supset B(0,r_n)$$ for some sequence $r_n$ with $r_n\to \infty$. Also choose an $R>0$ big enough so that $g(B(0,R_1))\subset B(0,R)$ while at the same time $R>R_1$, which implies that $$T(f^n(B(0,R)))\supset B(0,r_n).$$ 
		Pick now a $x_0 \in \mathbb{R}^d$ such that $g(x_0)\in {\mathcal{A}}(f)$. We will then show that $x_0\in{\mathcal{A}}(f)$. We know from \eqref{qfirst}, in other words the definition of the fast escaping set, that there exists an $L\in\mathbb{N}$ such that $$f^{n+L}(g(x_0))\not\in T(f^n(B(0,R))),  \hspace{2mm}\text{for all}\hspace{2mm} n\in \mathbb{N}.$$ 
		Since $f^{n+L}(g(x_0))=g(f^{n+L}(x_0))$ we will have that $$g(f^{n+L}(x_0))\not\in T(f^n(B(0,R))),  \hspace{2mm}\text{for all}\hspace{2mm} n\in \mathbb{N}.$$
		This together with the fact that $g(B(0,R_1))\subset B(0,R)$ implies that $$	g(f^{n+L}(x_0))\not\in T(f^n(g(B(0,R_1)))$$ and thus by permutability
		\begin{align}
		g(f^{n+L}(x_0))&\not\in T(g(f^n(B(0,R_1)))\label{fst}.
		\end{align}
		Assume now that there is a $n\in \mathbb{N}$ such that $f^{n+L}(x_0)\in T(f^n(B(0,R_1)))$ then $$g(f^{n+L}(x_0))\in g(T(f^n(B(0,R_1)))).$$
		But its true that \cite[Proposition 2.4]{B-D-F}  $g(T(f^n(B(0,R_1))))\subset T(g(f^n(B(0,R_1)))).$ Thus we would have that $g(f^{n+L}(x_0))\in T(g(f^n(B(0,R_1))))$ which contradicts (\ref{fst}). Hence its true that $$f^{n+L}(x_0)\not\in T(f^n(B(0,R_1))),  \hspace{2mm}\text{for all}\hspace{2mm} n\in \mathbb{N}$$ and thus $x_0 \in {\mathcal{A}}(f)$. Hence $g^{-1}({\mathcal{A}}(f))\subset{\mathcal{A}}(f).$ Since $g$ is an open map it easily follows that 
		$$g^{-1}(\overline{{\mathcal{A}}(f)})\subset \overline{g^{-1}({\mathcal{A}}(f))}\subset \overline{{\mathcal{A}}(f)}.$$
		Lastly, for the other half of the theorem we just change the roles of $f$ and $g$.
	\end{proof}
	The next lemma tells us that the Julia set of a quasiregular map is the smallest completely invariant closed set under $f$, which is a well known property of the Julia set in the complex plane. 
	\begin{lemma}\label{invariant}
		Let $f:\mathbb{R}^d\to \mathbb{R}^d$ be a quasiregular map. If $K$ is a closed set with $f(K)\subset K$ and $f^{-1}(K)\subset K$  and $\capac K>0$ then $\mathcal{J}(f)\subset K$.
	\end{lemma}
	\begin{proof}
		Take any neighbourhood, $U$, of a point $x\in{\mathcal{J}}(f), $ then by the definition of the Julia set $$\capac\left(\mathbb{R}^d\setminus\bigcup_{n=1}^{\infty}f^n(U)\right)=0.$$
		Hence, $ K\cap \bigcup_{n=1}^{\infty}f^n(U)\not=\emptyset.$
		This means that there is a $x_0\in U$ with $f^n(x_0)\in K$ for some $n\in \mathbb{N}$, and because $K$ is completely invariant under $f$ we will have that $x_0 \in K$. Hence, every neighbourhood, $U$ of a point in ${\mathcal{J}}(f)$ contains a point of $K$ and because $K$ is a closed set, this implies that ${\mathcal{J}}(f) \subset K$.
	\end{proof}
	\begin{proof}[Proof of Theorem \ref{my theorem}]
		First of all, since ${{\mathcal{A}}(f)}\subset {\mathcal{J}}(f)$ and since ${\mathcal{J}}(f)$ is closed we get that $\overline{{\mathcal{A}}(f)}\subset {\mathcal{J}}(f)$. As we have already mentioned, in the end of section 2, by \cite{B-F-N} we always know that $${\mathcal{J}}(f)\subset \partial{{\mathcal{A}}(f)}\subset \overline{{\mathcal{A}}(f)}.$$
		Hence we will have that $${\mathcal{J}}(f)=\overline{{\mathcal{A}}(f)}.$$
		Hence by Lemma \ref{lemma 1} we will have that $g({\mathcal{J}}(f))\subset {\mathcal{J}}(f)$ while from Lemma \ref{lemma2} we will have that $g^{-1}({\mathcal{J}}(f))\subset {\mathcal{J}}(f)$. This means that ${\mathcal{J}}(f)$ is completely invariant under $g$. Also from Theorem \ref{berg fast} we know that $\overline{{\mathcal{A}}(f)}$ contains continua, since its components are unbounded, and thus it cannot have zero capacity because zero capacity sets are totally disconnected (see \cite[Corollary III.2.5]{Rickman}) namely $$\capac\overline{{\mathcal{A}}(f)}=\capac{\mathcal{J}}(f)>0.$$
		Hence, we can now apply Lemma \ref{invariant}  and conclude that $\mathcal{J}(g)\subset \mathcal{J}(f)$.
		
		By a completely analogous argument we can also show that ${\mathcal{J}}(f) \subset {\mathcal{J}}(g)$ and thus ${\mathcal{J}}(g) ={\mathcal{J}}(f).$
	\end{proof}
	\begin{proof}[Proof of Theorem \ref{baker generalization}]
		We will prove that ${\mathcal{J}}(f)$ is completely invariant under $g$. We already know from Lemma \ref{lemma 1} that $g({\mathcal{J}}(f))\subset {\mathcal{J}}(f)$. Hence, it is enough to prove that $g^{-1}({\mathcal{J}}(f))\subset {\mathcal{J}}(f)$.\\\\
		Take a point $x_0\in\mathbb{R}^d$ such that $g(x_0)=\phi( f(x_0))\in {\mathcal{J}}(f)$. Take $V$ be any neighbourhood of $x_0$, then $U=\phi(f(V))=\{\phi(f(x)):x\in V\}$ is a neighbourhood of $\phi(f(x_0))$. Hence, 
		\begin{align}\label{eqbaker}
		\capac\left(\mathbb{R}^d\setminus\bigcup_{k=0}^{\infty}f^k(U)\right)=0.
		\end{align}
		But, it is true that $f\left(\phi^{-1}(U)\right)=\phi^{-1}\left(f(U)\right)$. Indeed, using the fact that $f$ commutes with $\phi\circ f$, \begin{align*}
		f\left(\phi^{-1}(U)\right)&=f\left(\phi^{-1}(\phi(f(V)\right)=f(f(V))\\&=\phi^{-1}(\phi(f(f(V)))=\phi^{-1}(f(\phi(f(V)))\\&=\phi^{-1}(f(U)).
		\end{align*}	 
		This easily implies that 
		\begin{equation}\label{eq2baker}
		f^n\left(\phi^{-1}(U)\right)=\phi^{-1}(f^n(U)),\hspace{1mm} \text{for all} \hspace{1mm}n\in\mathbb{N}.
		\end{equation} 
		By using (\ref{eqbaker}) and (\ref{eq2baker}) now, we conclude that \[\capac\left(\mathbb{R}^d\setminus\bigcup_{k=0}^{\infty}\phi\left(f^k\left(\phi^{-1}(U)\right)\right)\right)=\capac\left(\mathbb{R}^d\setminus\bigcup_{k=0}^{\infty}f^k(U)\right)=0.\]
		But it is true that \[\mathbb{R}^d\setminus\bigcup_{k=0}^{\infty}\phi\left(f^k\left(\phi^{-1}(U)\right)\right)=\phi\left(\mathbb{R}^d\setminus\bigcup_{k=0}^{\infty}f^k\left(\phi^{-1}(U)\right)\right).\] 
		Hence, by using the $K_I$-inequality (Theorem \ref{KI ineq}), we conclude that
		\[\capac\left(\mathbb{R}^d\setminus\bigcup_{k=0}^{\infty}f^k\left(\phi^{-1}(U)\right)\right)=0.\]
		In other words, \[\capac\left(\mathbb{R}^d\setminus\bigcup_{k=0}^{\infty}f^k(f(V))\right)=0,\]which implies that \[\capac\left(\mathbb{R}^d\setminus\bigcup_{k=0}^{\infty}f^k(V)\right)=0.\]Thus $x_0\in {\mathcal{J}}(f)$. By a similar argument we can also prove that ${\mathcal{J}}(g)$ is invariant under $f$.\\\\
		Now, since we know that $\capac{\mathcal{J}}(f)>0$ and  $\capac{\mathcal{J}}(g)>0$, we can apply Lemma \ref{invariant} twice and conclude that $\mathcal{J}(f)=\mathcal{J}(g)$.  
	\end{proof}
	
	\section{Proof of Theorem \ref{baker generalization 2}}
	In our proof of Theorem \ref{baker generalization 2} we will need the notion of a function having the \textit{pits effect}. This concept was first introduced by Littlewood and Offord in \cite{Littlewood} and a variant of it was used by Bergweiler and Nicks in \cite{B-Nicks} in order to develop an iteration theory for quasiregular maps of transcendental type. In what follows with $|\cdot|$ we denote the usual euclidean norm. 
	\begin{definition}
		A quasiregular map $f:\mathbb{R}^d\to\mathbb{R}^d$  of transcendental type is said to have the \textit{pits effect} if there exists $N\in \mathbb{N}$ such that, for all $\alpha>1$, for all $\lambda>1$ and all $\varepsilon>0$	there exists $R_0$ such that if $R>R_0$, then $$\{x\in\mathbb{R}^d:R\leq|x|\leq\lambda R,\hspace{1mm} |f(x)|\leq R^\alpha\}$$ can be covered by $N$ balls of radius $\varepsilon R$.
	\end{definition}
	We must also mention here that in \cite{B-Nicks} the authors first define the pits effect using the condition $|f(x)|\leq 1$ instead of $|f(x)|\leq R^\alpha$ and later prove that those two are actually the same \cite[Theorem 8.1]{B-Nicks}. 
	\begin{lemma}\label{pits}
		Let $f:\mathbb{R}^d\to \mathbb{R}^d$ and $g:\mathbb{R}^d\to \mathbb{R}^d$ be permutable quasiregular maps. Assume that $g=f+c$, where $c\not=0$ is a constant in $\mathbb{R}^d$. Then $f$ does not have the pits effect.
	\end{lemma}
	\begin{proof}
		For any $N\in\mathbb{N}$, we will find a sequence $R_m\to \infty$ and  $\lambda>1$, $\varepsilon>0$, $\alpha >1$ such that	\[A=\{x\in\mathbb{R}^d:R_m\leq|x|\leq\lambda R_m,\hspace{1mm} |f(x)|\leq R_m^\alpha\}\] cannot be covered by $N$ balls of radius $\varepsilon R_m$.\\\\
		First pick a $N\in\mathbb{N}$. Choose also a point $x_0\in\mathbb{R}^d$ that lies in the half-line connecting $0$ with $\infty$ and passes through $c$. Hence the sequence  $|x_0+nc|, n\in\mathbb{N}$ is an increasing sequence. Also since the number of omitted values of $f$ is finite we can assume that this half line does not contain any omitted values from the point $x_0-c$ onwards. We now set $R_m=|x_0+mc|$ and  we will show that a segment of the half line is contained in $A$ and that it is not possible to cover it with $N$ balls. Choose $\varepsilon=1/10$, then with $N$ balls of radius $R_m/10$ we can cover distance at most $\frac{NR_m}{5}$. Hence, if we take  
		\[(\lambda-1) R_m>\frac{NR_m}{5}\Leftrightarrow\lambda>\frac{N}{5}+1,\] then we cannot cover the part of the half line that lies between the circles with radius $R_m$ and $\lambda R_m$ with those $N$ balls. Now we only need to show that this part of the half line also satisfies the other condition; Namely that $|f(x)|\leq R_m^{\alpha}$ for some $\alpha>1$.  Observe that all the points on the half line, after $x_0$, can be written as $y+nc$ for some $y$ on the line segment from $x_0-c$ to $x_0$ and some $n\in\mathbb{N}$. Then since those points are not omitted by $f$ we have that   \[f(y+nc)=f(y+(n-1)c+c)=f(f(w_n)+c),\] $\hspace{1mm} \text{for some}\hspace{1mm} w_n\in\mathbb{R}^d\hspace{1mm} \text{with}\hspace{1mm} f(w_n)=y+(n-1)c$. Thus thanks to the fact that $f$ is commuting with $f+c$ we get that \[f(y+nc)=f(f(w_n))+c=f(y+(n-1)c)+c.\]By repeating this argument $n$ times we get that \[f(y+nc)=f(y)+nc,\hspace{1mm}\text{for all}\hspace{1mm}n\in\mathbb{N}.\] Hence for any point, $y+nc$, on the half line that lies between the circles with radius $R_m$ and $\lambda R_m$ we have that 
		\begin{align*}
		|f(y+nc)|=|f(y)+nc|&\leq|f(y)-y|+|y+nc|\\&\leq|f(y)-y|+\lambda R_m.\end{align*} If we now take $\alpha=2$  we will have $|f(y)-y|+\lambda R_m\leq R_m^\alpha$, for all big enough $R_m$ and thus the second condition will hold for all points on this line segment.
	\end{proof}
	\begin{lemma}\label{rotation}
		Let $f:\mathbb{R}^d\to\mathbb{R}^d$ be a quasiregular map of transcendental type and $L(x)=~aUx$ a linear map where $a\in\mathbb{R}\setminus \{0\}$ and $U\in SO(d)$ ($SO(d)$ is the special orthogonal group). If $g=L\circ f+c$, where $c\in \mathbb{R}^d$, commutes with $f$ then $|a|=1$.
	\end{lemma}
	\begin{proof}
		Without loss of generality let us assume, towards a contradiction, that $|a|>1$. Otherwise we can just consider $g$ and $\frac{1}{a}U^{-1}g-c$ which have the required form. Pick a large positive number $r'>0$. Then there is a $y_{r'}\in\mathbb{R}^d$ with $|y_{r'}|=r'$ such that $M(r',f)=|f(y_{r'})|,$ where $M(r',f)=\max_{|z|=r'}\{|f(z)|\}$. Now by Rickman's generalization of Picard's theorem, the fact that $f$ is a transcendental quasiregular map and the fact that $L$ is injective there is a point $x_{r'}\in\mathbb{R}^d$ such that $y_{r'}=(L\circ f)(x_{r'})+c$ and thus \[M(r',f)=|f(y_{r'})|=|f((L\circ f)(x_{r'})+c)|.\]  We set $r=|f(x_{r'})|$. Note that \[r=\left|L^{-1}(y_{r'}-c)\right|\geq |y_{r'}|\left|L^{-1}\left(\frac{y_{r'}}{|y_{r'}|}\right)\right|-|L^{-1}(c)|\geq \frac{r'}{|a|}-|L^{-1}(c)|. \]  Thus $r\to\infty$ as $r'\to \infty$. Also note that \[r'=|y_{r'}|=\left|L(f(x_{r'}))+c\right|\geq |a|r-|c|.\]  Hence, if we take a $1<\lambda<|a|$ then for all large enough $r'$ we have that $r'\geq\lambda r$. From the fact obey the maximum modulus principle we conclude that  $M(r,f)$ is an increasing function of $r$. Hence \begin{align*}\frac{M(\lambda r,f)}{M(r,f)}&\leq\frac{M(r',f)}{M(r,f)}=\frac{|f(L(f(x_{r'}))+c)|}{M(r,f)}\\&=\frac{|L( f(f(x_{r'})))+c|}{M(r,f)}\\&\leq\frac{|a|\left|f(f(x_{r'}))\right|+|c|}{M(r,f)}\\&\leq\frac{M(r,f)|a|+|c|}{M(r,f)}.\end{align*}
		Hence $\frac{M(\lambda r,f)}{M(r,f)}$ stays bounded as $r\to \infty$, which is a contradiction since this ratio tends to infinity as $r\to \infty$ (see \cite[Lemma 3.3]{Berg2006}). Hence $|a|=1$ .
	\end{proof}
	
	\begin{proof}[Proof of Theorem \ref{baker generalization 2}]
		First, from Lemma \ref{rotation} we will have that $a=1$. From Lemma \ref{pits} we will have that $f$ does not have the pits effect. Hence, from \cite[Corollary 1.1]{B-Nicks} we have that $\capac {\mathcal{J}}(f)>0$. Also note here that if $f$ does not have the pits effect then, by the definition, $f+c$ also does not. This again implies that $\capac {\mathcal{J}}(g)>0$. Now we can apply Theorem \ref{baker generalization} and conclude that ${\mathcal{J}}(f)={\mathcal{J}}(g)$.
	\end{proof}
	\section{Proof of Theorem \ref{theorem2}}
	We will prove first that $O_f^{-}(\infty)\subset O_g^{-}(\infty)$. To that end, let $x_0\in \mathbb{R}^d$ be a point in $O_f^{-}(\infty)$. This means that $f^n(x_0)=\infty$ for some $n\in \mathbb{N}$. This in turn implies that $f^{n-1}(x_0)$ is a pole of $f$. Now, note that $f$ and $g$ must have the same poles since if $z_0$ is a pole of $f$ but not $g$ then $g\circ f$ has an essential singularity in $z_0$ while $f\circ g$ does not, thus $f\circ g\not= g\circ f$ on a punctured neighbourhood of $z_0$. Hence, we will also have that $g(f^{n-1}(x_0))=\infty$. By using the fact that $f$ commutes with $g$ we have that $f(g(f^{n-2}(x_0)))=\infty$ and thus $ g(f^{n-2}(x_0))$ is a pole of $f$ which again implies it is also a pole of $g$. Using this argument $n$ times yields that $g^n(x_0)=\infty$ and thus $x_0\in  O_g^{-}(\infty)$.\\\\
	The other inclusion follows similarly by switching the roles of $f$ and $g$. Hence we have that $O_f^{-}(\infty)= O_g^{-}(\infty)$ and thus ${\mathcal{J}}(f)={\mathcal{J}}(g)$.
	
	\section{Proof of Theorems \ref{examples}, \ref{theorem1}}
	\begin{proof}[Proof of Theorem \ref{examples}]
		We want to construct uncountably many quasiregular maps that commute with a specific entire function. In order to do that we will follow the example given in \cite[section 2]{B-H} where the authors construct uncountably many continuous functions $g$ that commute with the function $f(z)=c(e^{z^2}-1)$, where $c$ is a large positive number. Note that $f$ has a superattracting fixed point at $0$ and there is a conformal function $\phi$, from the immediate basin of attraction $A$ of $f$ to the unit disk, that conjugates $f$ with $z^2$. That map $\phi$ is in fact quasiconformal on an open set that contains $A$ and that is due to the fact that $f$ is a polynomial-like mapping (see \cite{B-H} and references therein for more details). In order to construct the required map they first define a function $G$ which commutes with $z^2$. \\\\
		We define $G:\mathbb{C}\to\mathbb{C}$ by $G(z)=|z|^{m-1}z$ for some positive real number $m$. As we can easily see $G$ commutes with $z\mapsto z^2$. In fact $G$ is well known to be $M$-quasiconformal with $M=\max\{m,1/m\}$. Also, note that $G(z)=z$, for $|z|= 1$.
		Because we have uncountably many choices for $m$ we also have uncountably many such maps $G$.\\\\
		Next, by defining $g(z)=\phi^{-1}(G(\phi(z)))$ for $z\in A$ we see that \[g(f(z))=\phi^{-1}(G(\phi(f(z))))=\phi^{-1}(G(\phi(z)^2))=f(\phi^{-1}(G(\phi(z))))=f(g(z)), z\in A.\]In order to extend $g$ to the whole plane we argue as follows. If $B$ is a component of the basin of attraction of the fixed point at $0$ then there is a $n$ with $f^n(B)=A$. Choose $n$ to be the minimal with that property and define $g(z)=f^{-n}(g(f^n(z)))$, for all $z\in B$, with the appropriate branch of $f^{-n}$. Since $G$ coincides with the identity map on the unit circle and $\phi$ extends continuously and bijectively on $\partial A$  then we can extend $g$ to $\partial A\cup \partial B$ by setting $g(z)=z$ there. We can now extend $g$ to the rest of the plane, the complement of the basin of attraction, by setting $g(z)=z$ there. The functions $g:\mathbb{C}\to\mathbb{C}$ we get through this method  will be quasiregular and because our choices for $G$ are uncountably many, so are our functions $g$. 
	\end{proof}
	For the proof of Theorem \ref{theorem1} we will need a lemma first which is a generalization of a theorem of Polya in \cite{Polya1926}.
	\begin{lemma}\label{lemmasix}
		Let $f:\mathbb{R}^d\to\mathbb{R}^d$ and $g:\mathbb{R}^d\to\mathbb{R}^d$ be  quasiregular maps. Then \[M(r,f\circ g)\geq M(cM(r/2,g),f),\] where $c$ is a constant and $r$ is a large enough positive real number.
	\end{lemma}
	\begin{proof}
		From  \cite[Lemma 4.1]{BFLM} we know that if $D(r)$ denotes the ball of radius $r$ centred at $0$ and $S(r)$ the corresponding sphere, then for large enough $r$ there exists an $L\geq c M(r/2,g)$ such that $S(L)\subset g(D(r))$, where $c$ is a constant that does not depend on $L$ or $r$. By the choice of $L$ and the maximum modulus principle we have that \[M(cM(r/2,g),f)\leq M(L, f) .\] Now since $S(L)$ is inside $g(D(r))$ we will have that $M(L, f)=|f(g(x_0))|$, for some $x_0$ in $D(r)$. Thus by the maximum modulus principle again we will have that \[M(cM(r/2,g),f)\leq M(r,f\circ g).\]
	\end{proof}

	\begin{proof}[Proof of Theorem \ref{theorem1}]
		Suppose that $f$ is of transcendental type. By assumption $g$ is of polynomial type and $\deg g>K(g)$. Due to the maximum modulus principle we have that 
		\begin{equation}\label{firstq}
		M(r,g\circ f)\leq M(M(r,f),g).
		\end{equation}
		We know from \cite[Theorem III.4.7]{Rickman} that for all $x$ with large enough $|x|$ it is true that  \[A|x|^{n_1}\leq|g(x)|\leq B|x|^{n_2}, \] where $A,B$ constants, $n_1=\left(\deg g/K_I(g)\right)^{\frac{1}{d-1}}$ and $n_2=\left(\deg g \cdot K_O(f)\right)^{\frac{1}{d-1}}.$ Note that by our assumptions $n_1>1.$
		This implies that for large enough $r>0$ we have that \[Ar^{n_1}\leq M(r,g)\leq Br^{n_2}, \hspace{2mm}\text{where} \hspace{2mm} n_1, n_2 >1.\]  Hence by (\ref{firstq}) we have that \begin{equation}\label{secondq}
		M(r,g\circ f)\leq B \cdot M(r,f)^{n_2}.
		\end{equation} By Lemma \ref{lemmasix} we now know that \[M(r,g\circ f)=M(r, f\circ g)\geq M(cM(r/2,g),f).\] Thus
		\begin{equation}\label{thirdq}
		M(r,g\circ f)\geq M\left(cA\left(\frac{r}{2}\right)^{n_1},f\right).
		\end{equation}
		By (\ref{secondq}) and (\ref{thirdq}) we now get $$M\left(cA\left(\frac{r}{2}\right)^{n_1},f\right)\leq B \cdot M(r,f)^{n_2}.$$  Taking logarithms and rearranging we get that \[\frac{\log M(cA(\frac{r}{2})^{n_1},f)}{\log M(r,f)}\leq \log B \cdot n_2.\] But for large enough $r$ it is true that $cA(\frac{r}{2})^{n_1}\geq Q r$, where $Q>1$ is some constant. Thus  the previous inequality implies that \[\limsup_{\rho\to\infty}\frac{\log M(Q\rho,f)}{\log M(\rho,f)}<\infty.\] This  contradicts  \cite[Lemma 3.3]{Berg2006}.
	\end{proof}
	\section{Examples}
	Let us now give some examples of permutable maps and confirm that they have the same Julia set. First we give examples of holomorphic functions and then their counterparts in higher dimensions. We note here that the first three classes of examples are essentially the only ones possible in the case of rational functions that do not share a common iterate. The problem is still open in the case where the functions share an iterate. See \cite{er2, Ri2} for more details.
	
	\begin{remark1}\hfill
		\begin{enumerate}
			\item \textup{Consider the family of functions $f_n(z)=z^n, n\geq 2$. Obviously $f_n\circ f_m=f_{nm}=f_m\circ f_n.$ We can also easily see that ${\mathcal{J}} (f_n)=S^1,$ for all $ n\geq 2$, where $S^1$ denotes the unit circle. Thus ${\mathcal{J}}(f_n)={\mathcal{J}}(f_m)$.}
			
			\item \textup{Consider the family of Tchebycheff polynomials that satisfy $T_n(\cos z)=\cos(nz), n\geq 2$. It is easy to see that $T_n\circ T_m=T_m\circ T_n$. Also each of the Tchebycheff polynomials has as a Julia set the interval $[-2,2]$ (see \cite{carleson-gamelin} p. 30), so clearly  ${\mathcal{J}}(T_n)={\mathcal{J}}(T_m)$.}
			\item \textup{The family of \textit{Latt\`{e}s} maps provides another example of commuting functions. A rational map, of degree at least two, of the form \[f=\Theta\circ L\circ \Theta^{-1}\] is called Latt\`es. Here $L$ is an affine self map of the torus $\mathbb{C}/ \Lambda$, where $\Lambda\subset \mathbb{C}$ is a lattice of rank two, and $\Theta$ is a holomorphic map from the torus to $\hat{\mathbb{C}}$. One possible option is to choose $L(z)=az$ for any $a\in \mathbb{Z}[i]=\{x+yi:x,y\in\mathbb{Z}\}$, $|a|\geq2$ and $\Theta=\wp^2(z),$ where $\wp(z)$ is the Weierstrass elliptic function with periods $1$ and $i$. Then the Latt\`es maps that we take for the different values of $a$ are commuting. Also it is well known that the Julia set of Latt\`es maps is the entire Riemann sphere. For more details on Latt\`es maps we refer to the survey \cite{milnor-lattes}.}
			
			\item \textup{Consider an entire periodic function $P:\mathbb{C}\to \mathbb{C}$, with period $c\in \mathbb{C}$. Take $f(z)=P(z)+z$ and $g(z)=P(z)+z+c$. Then $f,g$ are permutable. Using now  a result of Baker \cite[Lemma~4.5]{Baker2}, which we generalized in Theorems \ref{baker generalization} and \ref{baker generalization 2} we conclude that ${\mathcal{J}}(f)={\mathcal{J}}(g)$. }
		\end{enumerate}	
	\end{remark1}
	\begin{remark2}\hfill
		\begin{enumerate}
			\item \textup{In \cite{Mayer2, Mayer1} Mayer constructs uniformly quasiregular analogues of the power maps, of Tchebycheff polynomials and of Latt\`es type maps which can be easily seen, just like in the complex case, that are permutable. Also those families of maps have the same Julia sets: the unit sphere, the unit disc and $\overline{\mathbb{R}^d}$ respectively.}
			
			\item \textup{There is a quasiregular analog of the exponential map in the complex plane called the Zorich map which was first defined by Zorich in \cite{Zorich}. For simplicity assume we work on $\mathbb{R}^3$ and denote this map by $Z:\mathbb{R}^3\to \mathbb{R}^3$. This map is periodic, with period 4, in its first two variables. In \cite{N-S} Nicks and Sixsmith defined a quasiregular map $g:\mathbb{R}^3\to \mathbb{R}^3$ of transcendental type such that
				\[ 
				g= \left\{
				\begin{array}{ll}
				Z+Id & x_3>L \\
				Id & x_3<0
				\end{array} ,
				\right. 
				\]
				where $Id$ the identity map and $L>0$ is a constant. By its construction this map satisfies $g(x_1+4, x_2, x_3)=g(x_1,x_2,x_3)+(4,0,0)$ for $0\leq x_3\leq L$ and hence for all $x_3$ (see \cite[section 6]{N-S} for details).
				Now define the function $f(x_1,x_2,x_3)=g(x_1,x_2,x_3)+(4,0,0)$. It is quite easy to see that $f$ commutes with $g$. Hence, by applying Theorem~\ref{baker generalization 2} we conclude that ${\mathcal{J}}(f)={\mathcal{J}}(g)$.}
			
			\item \textup{Another example is provided by \cite[section 7]{N-S} where the authors define the map \[h(x_1,x_2,x_3) = g(x_1,x_2,x_3)-(0,0,L'),\] where $g$ is the function of the previous example and $L'>0$ is a large constant. They also prove that $\mathcal{A}$$(h)\subset \mathcal{J}$$(h)$ and is quite easy to see that $h$ commutes with $h+(4,0,0)$. Hence, in this example we can apply Theorem \ref{my theorem} and conclude that the two functions have the same Julia set.}
		\end{enumerate}
	\end{remark2}
\bibliographystyle{amsplain}
\bibliography{bibliography}

\providecommand{\bysame}{\leavevmode\hbox to3em{\hrulefill}\thinspace}
\providecommand{\MR}{\relax\ifhmode\unskip\space\fi MR }
\providecommand{\MRhref}[2]{%
  \href{http://www.ams.org/mathscinet-getitem?mr=#1}{#2}
}
\providecommand{\href}[2]{#2}
\begin{thebibliography}{10}

\bibitem{baker1958}
I.~N. Baker, \emph{Zusammensetzungen ganzer funktionen}, Math. Z. \textbf{69}
  (1958), 121--163.

\bibitem{baker1}
I.~N. Baker, \emph{Permutable entire functions}, Math. Z. \textbf{79} (1962),
  243--249.

\bibitem{Baker2}
\bysame, \emph{Wandering domains in the iteration of entire functions}, Proc.
  London Math. Soc. \textbf{49} (1984), 563--576.

\bibitem{BRS}
A.~M. Benini, P.~J. Rippon, and G.~M. Stallard, \emph{Permutable entire
  functions and multiply connected wandering domains}, Advances In Mathematics
  \textbf{287} (2016), 451--462.

\bibitem{survey}
W.~Bergweiler, \emph{Iteration of meromorphic functions}, Bull. Amer. Math.
  Soc. \textbf{29} (1993), 151--188.

\bibitem{Berg2006}
\bysame, \emph{Fixed points of composite entire and quasiregular maps}, Ann.
  Acad. Sci. Fenn. Math. \textbf{31} (2006), 523--540.

\bibitem{Berg1}
\bysame, \emph{Iteration of quasiregular mappings}, Comput. Methods Funct.
  Theory \textbf{10} (2010), 455--481.

\bibitem{berg2013}
\bysame, \emph{Fatou-{Julia} theory for non-uniformly quasiregular maps},
  Ergodic Theory Dynam. Systems \textbf{33} (2013), no.~1, 1--23.

\bibitem{B-H}
W.~Bergweiler and A.Hinkkanen, \emph{On semiconjugation of entire functions},
  Math. Proc. Cambridge Philos. Soc. \textbf{137} (1999), 641--651.

\bibitem{B-D-F}
W.~Bergweiler, D.~Drasin, and A.~Fletcher, \emph{The fast escaping set for
  quasiregular mappings}, Anal. Math. Phys. \textbf{4} (2014), 83--98.

\bibitem{BFLM}
W.~Bergweiler, A.~Fletcher, J.~Langley, and J.~Meyer, \emph{The escaping set of
  a quasiregular mapping}, Proc. Amer. Math. Soc. \textbf{137} (2009), no.~2,
  641--651.

\bibitem{B-F-N}
W.~Bergweiler, A.~Fletcher, and D.~Nicks, \emph{The {Julia} set and the fast
  escaping set of a quasiregular mapping}, Comput. Methods Funct. Theory.
  \textbf{14} (2014), 209--218.

\bibitem{B-Nicks}
W.~Bergweiler and D.A. Nicks, \emph{Foundations for an iteration theory of
  entire quasiregular maps}, Israel J. Math. \textbf{201} (2014), no.~1,
  147--184.

\bibitem{carleson-gamelin}
L.~Carleson and T.~W. Gamelin, \emph{Complex dynamics}, Universitext,
  Springer-Verlag, New York, 1993.

\bibitem{er2}
A.~E. Eremenko, \emph{On some functional equations connected with iteration of
  rational functions}, Leningrad Math J. \textbf{1} (1990), no.~4, 905--919.

\bibitem{fatou1919}
P.~Fatou, \emph{Sur les equations fonctionelles}, Bull. Soc. Math. France.
  \textbf{47-48} (1919, 1920), 161--271, 33--94, 208--314.

\bibitem{fatou1923}
\bysame, \emph{Sur l'iteration analytique et les substitutions permutables}, J.
  Math. Pures Appl. \textbf{2} (1923), 343--384.

\bibitem{fatou1926}
\bysame, \emph{Sur l'iteration des fonctions transcendantes entieres}, Acta
  Math. \textbf{47} (1926), 337--370.

\bibitem{FLETCHER2010}
Alastair~N. Fletcher and Daniel~A. Nicks, \emph{Quasiregular dynamics on the
  n-sphere}, Ergodic Theory Dynam. Systems \textbf{31} (2010), no.~1, 23--31.

\bibitem{iyer}
V.~Ganapathy~Iyer, \emph{On permutable integral functions}, J. London Math.
  Soc. \textbf{s1-34} (1959), no.~2, 141--144.

\bibitem{Julia1918}
G.~Julia, \emph{Sur l'iteration des fonctions rationelles}, J. Math. Pures
  Appl. \textbf{4} (1918), no.~7, 47--245.

\bibitem{julia1922}
\bysame, \emph{Memoire sur la permutabilite des fractions rationnelles}, Ann.
  Sci. Ecole Norm. Sup. \textbf{39} (1922), no.~3, 131--215.

\bibitem{Littlewood}
J.~E. Littlewood and A.~C. Offord, \emph{On the distribution of zeros and
  $\alpha$-values of a random integral function.}, Ann. of Math. (2)
  \textbf{49} (1948), 885--952.

\bibitem{Mayer2}
V.~Mayer, \emph{Uniformly quasiregular mapings of {Lattes} type}, Conform.
  Geom. Dyn. \textbf{1} (1997), 104--111.

\bibitem{Mayer1}
\bysame, \emph{Quasiregular analogs of critically finite rational functions
  with parabolic orbifold}, Journal D'Analyse Math. \textbf{75} (1998),
  105--119.

\bibitem{milnor-lattes}
J.~Milnor, \emph{On {Latt\`es} maps}, arXiv:math/0402147 (2004).

\bibitem{milnor}
\bysame, \emph{Dynamics in one complex variable}, Annals of Mathematics
  Studies, Princeton University Press, 2006.

\bibitem{N-S}
D.~A. Nicks and D.~J. Sixsmith, \emph{Periodic domains of quasiregular maps},
  Ergodic Theory Dynam. Systems \textbf{38} (2018), 2321--2344.

\bibitem{NICKS2017}
Daniel Nicks and David Sixmith, \emph{Hollow quasi-fatou components of
  quasiregular maps}, Math. Proc. Cambridge Philos. Soc. \textbf{162} (2017),
  no.~3, 561--574.

\bibitem{Osborne2016a}
J.~W. Osborne and D.~J. Sixsmith, \emph{On permutable meromorphic functions},
  Aequationes Math. \textbf{90} (2016), no.~5, 1025--1034.

\bibitem{Polya1926}
G.~P{\'{o}}lya, \emph{On an integral function of an integral function}, J.
  London Math. Soc. \textbf{s1-1} (1926), no.~1, 12--15.

\bibitem{Rickman1980}
S.~Rickman, \emph{On the number of omitted values of entire quasiregular
  mappings}, Journal d'Analyse Math. \textbf{37} (1980), 100--117.

\bibitem{rickman1985}
\bysame, \emph{The analogue of {Picard's} theorem for quasiregular mappings in
  dimension three}, Acta Math. \textbf{154} (1985), 195--242.

\bibitem{Rickman}
\bysame, \emph{Quasiregular mappings}, vol.~26, Ergeb. Math. Grenzgeb., no.~3,
  Springer-Verlag, Berlin, 1993.

\bibitem{R-S2}
P.~J. Rippon and G.~M. Stallard, \emph{On questions of {Fatou} and {Eremenko}},
  Proc. Amer. Math. Soc. \textbf{133} (2005), 1119--1126.

\bibitem{R-S}
\bysame, \emph{Fast escaping points of entire fucntions}, Proc. London Math.
  Soc. \textbf{105} (2012), no.~4, 787--820.

\bibitem{Ri1}
J.~F. Ritt, \emph{On the iteration of rational functions}, Trans. Amer. Math.
  Soc. \textbf{21} (1920), no.~3, 348--356.

\bibitem{Ri2}
\bysame, \emph{Permutable rational functions}, Trans. Amer. Math. Soc.
  \textbf{25} (1923), no.~3, 291--309.

\bibitem{vuorinen}
M.~Vuorinen, \emph{Conformal geometry and quasiregular mappings}, Lecture Notes
  in Math., vol. 1319, Springer-Verlag, Berlin, 1988.

\bibitem{WARREN2018}
L.~Warren, \emph{On the iteration of quasimeromorphic mappings}, Math. Proc.
  Cambridge Philos. Soc. \textbf{168} (2018), no.~1, 1--11, Advance online
  publication. doi: 10.1017/s030500411800052x.

\bibitem{Warren2019}
Luke Warren, \emph{Constructing a quasiregular analogue of $z \exp(z)$ in
  dimension 3}, arXiv:1907.04720 (Preprint).

\bibitem{Zorich}
V.~A. Zorich, \emph{The theorem of {M. A. Lavrent'ev} on quasiconformal
  mappings in space}, Mat. Sb. \textbf{74} (1967), 417--433.

\end{thebibliography}
\end{document}